\documentclass[12pt]{amsart}
\usepackage{amsmath}
\usepackage{amsfonts}
\usepackage{amssymb}
\usepackage{graphicx}
\usepackage{hyperref}
\usepackage{mathrsfs}
\usepackage{pb-diagram}
\usepackage{epstopdf}
\usepackage{caption}
\usepackage{subcaption}
\usepackage{hyperref}

\usepackage[all]{xy}
\usepackage{amscd}
\usepackage{pstricks}
\usepackage{verbatim}

\newtheorem{thm}{Theorem}[section]
\newtheorem{lemma}[thm]{Lemma}
\newtheorem{cor}[thm]{Corollary}

\newtheorem{prop}[thm]{Proposition}
\newtheorem{defn}[thm]{Definition}
\newtheorem{example}[thm]{Example}
\newtheorem{remark}[thm]{Remark}

\numberwithin{equation}{section}

\newcommand{\pairing}[2]{\left( #1\, , \, #2 \right)}

\newcommand{\bP}{\mathbb{P}}
\newcommand{\Z}{\mathbb{Z}}

\newcommand{\R}{\mathbb{R}}
\newcommand{\C}{\mathbb{C}}

\newcommand{\cP}{\mathcal{P}}

\newcommand{\CM}{\mathcal{M}}

\newcommand{\cb}{\mathfrak{b}}
\newcommand{\unu}{\underline{\nu}}

\newcommand{\ev}{\mathrm{ev}}
\newcommand{\pt}{\mathrm{pt}}
\newcommand{\sg}{\mathrm{sg}}
\newcommand{\tg}{\mathrm{tg}}
\newcommand{\cX}{\mathcal{X}}

\begin{document}

\title[GS slab and open GW]{Gross-Siebert's slab functions and open GW invariants for toric Calabi-Yau manifolds}

\author[S.-C. Lau]{Siu-Cheong Lau}
\address{Harvard University}
\email{s.lau@math.harvard.edu}

\begin{abstract}
This paper derives an equality between the slab functions in Gross-Siebert program and generating functions of open Gromov-Witten invariants for toric Calabi-Yau manifolds, and thereby confirms a conjecture of Gross-Siebert on symplectic enumerative meaning of slab functions.  The proof is based on the open mirror theorem of Chan-Cho-Lau-Tseng \cite{CCLT13}.  It shows an instance of correspondence between tropical and symplectic geometry in the open sector.
\end{abstract}

\maketitle

\section{Introduction}

The celebrated Gross-Siebert program \cite{GS07,GHK} reconstructs mirrors of algebraic varieties by using toric degenerations and tropical geometry.  It can be regarded as an algebraic version of the SYZ program \cite{SYZ96}, which gives a brilliant way of handling quantum corrections coming from singular strata.  In \cite{GS-torCY}, Gross and Siebert applied their construction to toric Calabi-Yau manifolds and construct their mirrors.  They also discussed enumerative meaning of their slab functions in terms of counting tropical discs.

On the other hand, SYZ construction for toric Calabi-Yau manifolds was carried out by a joint work of the author with Chan and Leung \cite{CLL} using open Gromov-Witten invariants in symplectic geometry.  Moreover, an equality with the Hori-Iqbel-Vafa mirror was proved in \cite{CLT11} for the total space of the canonical line bundle of a toric Fano manifold, and in \cite{CCLT13} for general toric Calabi-Yau orbifolds.  As a result mirror maps can be expressed in terms of open Gromov-Witten invariants, which was called the open mirror theorem.

It is natural to ask whether these two different approaches produce the same mirror for toric Calabi-Yau manifolds.  The aim of this paper is to give an affirmative answer to this question.  In particular, we establish an equality between Gross-Siebert normalized slab functions for toric Calabi-Yau manifolds and the wall-crossing generating function for open Gromov-Witten invariants.  This confirms the conjecture by Gross-Siebert \cite[Conjecture 0.2]{GS07} that their slab functions have an enumerative meaning in terms of counting holomorphic discs bounded by a fiber of a Lagrangian torus fibration.

\begin{thm} \label{thm:main}
For a toric Calabi-Yau manifold, let $f_{\cb,v_j}$ be the Gross-Siebert slab function for a slab $\cb$ and a primitive generator $v_j$ of a ray in the fan.  Let $n_\beta$ denotes the open Gromov-Witten invariants associated to a disc class $\beta \in \pi_2(X,L)$ bounded by a moment-map fiber $L$.  We have the equality
$$f_{\cb,v_j}(q,z) = \sum_{i=1}^m \left(\sum_{\alpha} q^\alpha n_{\beta_i+\alpha}\right) q^{C_i-C_j} z^{v_i-v_j} = \sum_{i=1}^m \exp g_i(\check{q}(q)) \cdot q^{C_i-C_j} z^{v_i-v_j}$$
where $C_i \in H_2(X)$ are some explicit curve classes (see Equation \eqref{eq:C}), $\check{q}(q)$ is the mirror map, and $g_i(\check{q})$ is an explicit hypergeometric series (see Definition \ref{def:g}).  The monomials $z^v$ for a vector $v$ is explained in Section \ref{sect:setup}.
\end{thm}

The second equality in the above theorem is the main result of the joint work \cite{CCLT13} of the author with Chan, Cho and Tseng.  In this paper we will derive the first equality based on the work of \cite{CCLT13} and the combinatorial meaning of hypergeometric series, see the proof of Proposition \ref{prop:key} and Section \ref{thm:main}.   As a consequence, this implies that the mirror constructed using tropical geometry by Gross-Siebert equals to that constructed using symplectic geometry by Chan-Lau-Leung.  This gives a correspondence between tropical and symplectic geometry in the open sector (while the correspondence in the closed sector has been well-studied, see for instance \cite{Mikhalkin,Gross-trop}).  The geometry underlying such a correspondence is very rich and deserves a further study.

As an immediate consequence,

\begin{cor} \label{cor:same}
For toric Calabi-Yau manifolds, the mirror constructed using tropical geometry by Gross-Siebert equals to the mirror constructed using symplectic geometry by Chan-Lau-Leung.
\end{cor}

\begin{remark}
In \cite{GS-torCY}, Gross-Siebert counted \emph{tropical discs} and trees and explained a conjectural relation with their slab functions.  This paper employs symplectic geometry instead of tropical geometry and expressse slab functions in terms of counting \emph{holomorphic discs}.  The statement and proof is independent of their tropical interpretation.
\end{remark}

\section*{Acknowledgement}

I am grateful to Nai-Chung Conan Leung for encouragement and useful advice during preparation of this paper.  I express my gratitude to Mark Gross and Bernd Siebert for useful explanations of their program, and to my collaborators Kwokwai Chan, Cheol-Hyun Cho, Hansol Hong, Sang-Hyun Kim, Hsian-Hua Tseng and Baosen Wu for their continuous support.  This project is supported by Harvard University.

\section{Setup and notations for toric Calabi-Yau manifolds} \label{sect:setup}

In this section we briefly recall the essential matarials for toric Calabi-Yau manifolds used in this paper.  

Let $N$ be a lattice of rank $n$ and $M$ be the dual lattice.  Fix a primitive vector $\unu \in M$.  Let $\sigma$ be a (closed) lattice polytope of dimension $n-1$ contained in the affine hyperplane $\{v \in N_\R: \unu(v) = 1\}$.  By choosing a lattice point $v$ in the lattice polytope $\sigma$, we have a lattice polytope $\sigma - v$ in the hyperplane $\unu^\perp_\R \subset N_\R$.  Let $\cP$ be a triangulation of $\sigma$ such that each maximal cell is a standard simplex.  Then by taking a cone over this triangulation, we obtain a fan $\Sigma$ supported in $N_\R$.  Then $X=X_\Sigma$ is a toric Calabi-Yau manifold, which means that the anti-canonical divisor $-K_X = \sum_{i=1}^m D_i$ is linearly equivalent to zero.

Let $m$ denote the number of lattice points lying in the polytope $\sigma$, and let $v_1,\ldots,v_m$ be generators of rays of the fan $\Sigma$(which are one-to-one corresponding to the lattice points in $\sigma$).  By reordering $v_i$'s if necessary, we assume $\{v_1,\ldots,v_n\}$ belongs to a cone of the fan.  In particular it forms a basis of $N$. 
Then we can write $v_i = \sum_{l=1}^n v_{i,l} v_l$ for $v_{i,l} \in \Z$, and $v_{i,l} = \delta_{il}$ if $i \in \{1,\ldots,n\}$.

Denote by $D_i$ the toric prime divisor corresponding to $v_i$.  Each toric prime divisor $D_i$ corresponds to a basic disc class $\beta_i \in \pi_2(X,T)$ bounded by a Lagrangian torus fiber $T \subset X$ (see \cite{CO} for detailed discussions on basic disc classes).  We have $\pi_2(X,T) = \Z \langle \beta_1,\ldots,\beta_m\rangle$ and the exact sequence
$$ 0 \to H^2(X) \to \pi_2(X,T) \to N \to 0. $$

Since $X$ is Calabi-Yau, we have $c_1(\alpha) := -K_X \cdot \alpha = \sum_{i=1}^m D_i \cdot \alpha = 0$ for any $\alpha \in H_2^(X)$.  For a disc class $\beta \in \pi_2(X,T)$, its Maslov index is $\mu(\beta) = \sum_{i=1}^m D_i \cdot \beta$ (see \cite{CO}).  In particular the basic disc classes have Maslov index two.

Define the curve classes \begin{equation} \label{eq:C}
C_i := \beta_{i} - \sum_{l=1}^n v_{i,l} \beta_l
\end{equation}
for $i=n+1,\ldots,m$.  Then $\{C_{n+1},\ldots,C_m\}$ forms a basis of $H_2(X)$.  The dual basis of $H^2(X)$ is $\{D_{n+1},\ldots,D_m\}$.  The corresponding K\"ahler parameters and mirror complex parameters are denoted as $q=(q_{n+1},\ldots,q_{m})=(q^{C_{n+1}},\ldots,q^{C_m})$ and $\check{q}=(\check{q}_{n+1},\ldots,\check{q}_{m}) = (\check{q}^{C_{n+1}},\ldots,\check{q}^{C_{m}})$ respectively.  We set $C_1 = \ldots = C_n = 0$, and so $\check{q}_i = q_i = q^{C_i} = 1$ for $i=1,\ldots,n$.

\begin{remark}
In the language of Gross-Siebert program, the choice of the basis $\{C_i:i=n+1,\ldots,m\}$ corresponds to the choice of a set of piecewise linear functions supported on $\Sigma$.  $C_i$ corresponds to the piecewise linear function which takes value $1$ on $v_i$, $-v_{i,l}$ on $v_l$ for $l\in\{1,\ldots,n\}$, and zero on $v_l$ for $l \not\in \{i\} \cup \{1,\ldots,n\}$.
\end{remark}

Now we define the mirror complex variables which are used to define the mirror of the toric Calabi-Yau manifold $X$.
Denote by $\hat{z}^{v_i}$ for $i=1,\ldots,m$ the monomial $\prod_{l=1}^n \hat{z}_l^{v_{i,l}}$ in $n$ variables $\hat{z}_1,\ldots,\hat{z}_n$.  Then $\hat{z}^{v_i} = \hat{z}_i$ for $i=1,\ldots,n$.  This gives a correspondence between lattice points in $N$ and monomials in $n$ variables.  On the other hand, we need a correspondence between lattice points in $N^{\perp \unu}$ and monomials in $n-1$ variables.

Fix $j \in \{1,\ldots,m\}$ and set $\bar{v}_l := v_l - v_j \in N^{\perp \unu}$ for all $l \in \{1,\ldots,m\}$.  Then $\bar{v}_i = \sum_{l=1}^n v_{i,l} \bar{v}_l$ (since $\sum_{l=1}^n v_{i,l} = 1$ by $\unu(v_k) = 1$ for all $k$).  Fix a basis $\{e_1,\ldots,e_{n-1}\}$ of $N^{\perp \unu}$.  Then $e_k = \sum_{l=1}^n e_{k,l} v_l$ for some $e_{k,l} \in \Z$, and we define $z_k = \prod_{l=1}^n z_l^{e_{k,l}}$ for $k=1,\ldots,n-1$.  Since $\bar{v}_l$ can be expressed as an integer combination of $e_k$'s, $\hat{z}^{\bar{v}_l}$ can be regarded as monomials in the $(n-1)$ coordinates $z_k$'s.  $z_k$ for $k=1,\ldots,n-1$ will be used as complex coordinates for the mirror of $X$.

In the next two sections we review the construction of mirror varieties by tropical geometry and symplectic geometry respectively.

\section{Gross-Siebert slab functions}

In \cite{GS07} Gross-Siebert developed a construction of mirror varieties by extracting tropical geometric data from a toric degeneration of a smooth algebraic variety.  In a recent paper \cite{GS-torCY} they applied their construction to toric Calabi-Yau manifolds.  They constructed a toric degeneration of a toric Calabi-Yau manifold, and use it to construct its mirror Calabi-Yau.  The survey paper \cite{GS-invitation} provides an excellent exposition of their program.
  
In the following we will just sketch their construction in a very brief way, and focus on the slab functions and the resulting mirror varieties which are the main subjects of this paper.  Some of the notations here are different from that in \cite{GS-torCY} for convenience of comparing with the SYZ mirror constructed from symplectic geometry discussed in the next section.

The Gross-Siebert program starts with a toric degeneration of an algebraic variety $X$, which roughly speaking is a family of varieties $\cX^t$ where $\cX^t \cong X$ for $t \not= 0$ and $\cX^0$ is a union of toric varieties (of the same dimension) glued along toric strata.  From the intersection complex of the central fiber $\cX^0$, an affine manifold $B$ with singularities $\Delta$ is constructed, together with some initial `walls' which are certain codimension-one strata attached with `slab functions'.  From this initial data, a structure of walls on $B$ can be constructed order-by-order by working on scattering diagrams.  Intuitively by gluing all the connected components of the complement of the walls in $B$ by using the slab functions, one obtain a family of varieties which is defined as the mirror.

In the case of a toric Calabi-Yau manifold, which can be regarded as a local piece of a compact Calabi-Yau variety, the construction simplifies drastically.  $B$ can be identified as $\R^n$ topologically, and there is exactly one wall $H$ which can be identified as $\R^{n-1} \times \{0\}$.  The discriminant locus $\Delta$ is contained in this wall, which can be topologically identified as the dual of the triangulation on the polytope $\sigma \subset N_\R^{\perp \unu} \cong \R^{n-1}$.  $H - \Delta$ consists of several connected components, which have one-to-one correspondence with the lattice points $v_i$ in $\sigma$.

Slabs and slab functions play a key role in constructing the mirror family.  In this case the slabs $\cb$ are identified as the maximal cells in the triangulation $\cP$ of $\sigma$.  Slab functions are elements in $\C[[q_{n+1},\ldots,q_m]][z_1^\pm, \ldots, z_{n-1}^\pm]$ attached to slabs; in this case we have a slab function $f_{\cb,v} \in \C[[q_{n+1},\ldots,q_m]][z_1^\pm, \ldots, z_{n-1}^\pm]$ attached to each pair $(\cb,v)$, where $\cb$ is a slab and $v$ is a vertex of $\cb$.  The slab functions are determined by the following properties.

\begin{defn}[Slab functions] \label{def:slab}
Slab functions are elements in the collection $\{f_{\cb,v}\}_{(\cb,v)}$, where $f_{\cb,v} \in \C[[q]][z_1^\pm, \ldots, z_{n-1}^\pm]$, determined by the following properties:
\begin{enumerate}
\item The constant term of each $f_{\cb,v}$ (as a series in $q$ and $z$) is $1$.
\item If $v_i$ and $v_j$ are adjacent vertices of $\cb$, then
$$ f_{\cb,v_i} = q^{C_j-C_i} z^{v_j-v_i} f_{\cb,v_j}. $$
\item $\log f_{\cb,v}$ has no term of the form $a \cdot q^{C}$ where $a \in \C^\times$ and $C \in H_2(X)-\{0\}$.
\item If $v \in \cb \cap \cb'$, then $f_{\cb,v} = f_{\cb',v}$.  Hence  $f_{\cb,v}$ actually does not depend on $\cb$.
\end{enumerate}
\end{defn}
(3) is the most essential condition.  It is called to be the normalization condition.  Roughly speaking it means $\log f_{\cb,v}$ is counting tropical discs, and hence terms corresponding to tropical curves cannot appear.  From these four conditions one can compute the slab functions order-by-order.

Having the slab functions $f_{\cb',v}$, the outcome of the construction of Gross-Siebert program in this case is the following:

\begin{thm}[\cite{GS-torCY}]
The Gross-Siebert mirror of a toric Calabi-Yau manifold $X$ is
\begin{equation} \label{eq:X^tg}
\check{X}^{\tg} = \{(u,v,z) \in \C^2 \times (\C^\times)^{n-1}: uv = f_{\cb,v_j}(q,z) \}
\end{equation}
for a chosen $j \in \{1,\ldots,m\}$.  The superscript `tg' stands for tropical geometry.
\end{thm}
By Property (2) and (4) of Definition \ref{def:slab}, the varieties for different $j$'s are isomorphic to each other, and so we do not need to worry about the choice of $j$.

\section{Open GW invariants and SYZ mirrors of toric Calabi-Yau manifolds}

In \cite{CLL}, an SYZ construction of the mirror of toric Calabi-Yau manifolds was developed based on symplectic geometry.  This section reviews the construction briefly and states the end result.

The construction starts with a Lagrangian torus fibration of a toric Calabi-Yau manifold, which was constructed by Goldstein \cite{goldstein} and Gross \cite{gross_examples} independently.  The base $B$ of the fibration can be identified topologically as the upper half space $\R^{n-1} \times \R_{\geq -C}$ for a constant $C > 0$.  The codimension-two discriminant locus $\Delta$, like in the construction in the last section, is contained in the hyperplane $H = \R^{n-1} \times \{0\}$ and can be topologically identified as the dual of the triangulation on the polytope $\sigma \subset N_\R^{\perp \unu} \cong H$.  The connected components of $H - \Delta$ have one-to-one correspondence with the lattice points $v_i$ in $\sigma$.

We fix a choice of $j \in \{1,\ldots,m\}$ to define the following contractible open subset of $B$:
$$ U_j := B - \Delta - \cup_{i\not=j} (\textrm{ connected component of $(H-\Delta)$ corresponding to $v_i$}). $$
The torus fibration trivializes over $U_j$, and hence can be identified with $U_j \times N_R/N$.  Then $\pi_2(X,F_r)$ for any fiber $F_r$, $r \in U_j$ can be identified with $\Z^m$.

One essential ingredient of the construction is \emph{open Gromov-Witten invariant} of a Lagrangian fiber, which is roughly speaking counting holomorphic discs bounded by a Lagrangian fiber.  It turns out that open Gromov-Witten invariants of fibers at points above and below the wall $H$ differ drastically (and the invariants are not well-defined for fibers at the wall), and this is known as wall-crossing phenomenon which was studied by \cite{auroux07} in this context.  The invariants for fibers below the wall are only non-trivial for exactly one disc class, while the invariants for fibers above the wall can be identified with that of moment-map fibers.

\begin{defn}[Open Gromov-Witten invariants] \label{def:oGW}
Let $X$ be a toric Calabi-Yau manifold, $T$ a regular moment-map fiber, and $\beta \in \pi_2(X,T)$ a disc class bounded by the fiber $T$.  Let $\CM_1(\beta)$ be the moduli space of stable discs with one boundary marked point representing $\beta$.  The open Gromov-Witten invariant associated to $\beta$ is
$$ n_\beta = \int_{\CM_1(\beta)} \ev^*[\pt] $$
where $\ev:\CM_1(\beta) \to T$ is the evaluation map at the boundary marked point.
\end{defn}
By dimension counting, the open Gromov-Witten invariants are non-zero only when $\beta$ has Maslov index two.  Moreover stable disc classes of a moment-map fiber must take the form $\beta + \alpha \in \pi_2(X,T)$, where $\beta$ is a basic disc class and $\alpha$ is an effective curve class (see \cite{FOOOT}).

Having the open Gromov-Witten invariants $n_\beta$, the outcome of the construction in \cite{CLL} is the following:

\begin{thm} \label{thm:SYZ}
The SYZ mirror of a toric Calabi-Yau manifold is
\begin{equation} \label{eq:X^sg}
\check{X}^{\sg} = \left\{(u,v,z_1, \ldots, z_{n-1}) \in \C^2 \times (\C^\times)^{n-1}: uv = f_j(q,z) \right\},
\end{equation}
for $j \in \{1,\ldots,m\}$ chosen in the construction, where
\begin{equation} \label{eq:f}
f_j(q,z) = \sum_{i=1}^m \left(\sum_{\alpha} q^\alpha n_{\beta_i+\alpha}\right) q^{C_i-C_j} z^{v_i-v_j},
\end{equation}
The superscript `sg' stands for symplectic geometry.
\end{thm}

$f_j$ given in Equation \eqref{eq:f} is called to be the wall-crossing generating function of open Gromov-Witten invariants because it describes the change in generating function of open invariants when crossing the wall $H$.

\section{Open mirror theorem}

In \cite{CCLT13} an equality between open Gromov-Witten invariants and mirror map was derived, which is a main ingredient in our proof of Theorem \ref{thm:main}.  In this section we recall the result, which was called the open mirror theorem.

The main object of study of the open mirror theorem is the generating function of open Gromov-Witten invariants (see Definition \ref{def:oGW}).

\begin{defn}[Generating function of open Gromov-Witten invariants] \label{def:delta}
Let $X$ be a toric Calabi-Yau manifold.  The generating function of open Gromov-Witten invariants associated to a toric divisor $D_i$ is defined to be
$ \sum_{\alpha} q^\alpha n_{\beta_i + \alpha}$ which takes the form $1 + \delta_i (q)$.  The summation is over all curve classes $\alpha$.
\end{defn}

Mirror map plays a central role in mirror symmetry.  It gives a canonical local isomorphism between the K\"ahler moduli and the mirror complex moduli.  In our context, the generating functions defined above are expressed in terms of the mirror map.

\begin{defn}[Mirror map] \label{def:g}
For $j=1,\ldots,m$, define the hypergeometric functions
\begin{equation}\label{eq:g}
g_j(\check{q}):=\sum_{d}\frac{(-1)^{(D_j\cdot d)}(-(D_j\cdot d)-1)!}{\prod_{p\neq j} (D_p\cdot d)!}\check{q}^d
\end{equation}
where the summation is over all effective curve classes $d\in H_2^\text{eff}(X)$ satisfying
$$-K_X\cdot d=0, D_j\cdot d<0 \text{ and } D_p\cdot d \geq 0 \text{ for all } p\neq j.$$

Then the mirror map is defined as
$$ q_l = \check{q}_l \exp \left(-\sum_{k=1}^{m} \pairing{D_k}{C_l} g_l(\check{q})\right). $$
\end{defn}

Now we are prepared to state the open mirror theorem:

\begin{thm}[Open mirror theorem \cite{CCLT13}] \label{thm:open=g}
For a toric Calabi-Yau manifold,
$$ \sum_{\alpha} q^\alpha n_{\beta_i + \alpha} = \exp g_i(\check{q}(q)) $$
where $\check{q}(q)$ is the inverse of the mirror map in Definition \ref{def:g}.
\end{thm}

The proof involves identification of the generating functions of open invariants with certain generating functions of closed invariants of a toric compactification of $X$, which depends on the basic disc class $\beta_i$.  Then the closed invariants can be expressed in terms of hypergeometric functions and mirror maps by using the closed toric mirror theorem \cite{givental98,LLY1,LLY2,LLY3}.  Note that the compactification may contain orbifold strata, and so a natural class to study for this purpose is toric Calabi-Yau orbifolds.  Nevertheless, the theorem itself can be stated within the manifold setting.

\begin{remark}
Similar result holds for toric semi-Fano manifolds \cite{CLLT12}, in which the proof involves Seidel representation and degeneration techniques.  Open mirror theorem is useful for studying global mirror symmetry, as illustrated in \cite{CCLT12,L13}.
\end{remark}

\section{Proof of Theorem \ref{thm:main}}
We are now ready to prove the main Theorem \ref{thm:main}. The strategy is the following.  First wall-crossing generating function of open Gromov-Witten invariants $f_i$ can be expressed in terms of hypergeometric functions and mirror maps by using Theorem \ref{thm:open=g} (the main result of \cite{CCLT13}).  Then we prove that it satisfies the defining properties of slab functions by observing a combinatorial meaning of hypergeometric functions (as counting tropical curves, roughly speaking), and hence the equality between $f_i$ (see Equation \eqref{eq:f}) and $f_{\cb,v_i}$ (see Definition \ref{def:slab}) follows.  This implies Gross-Siebert mirror constructed from tropical geometry equals to the SYZ mirror constructed from symplectic geometry for toric Calabi-Yau manifolds.

First we do a change of coordinates in $z$, such that $f_j$ in the new coordinates has an explicit closed expression in terms of the mirror complex parameters $\check{q}$ (rather than $q$):
\begin{lemma} \label{lem:change}
Let $f_j$ be the wall-crossing generating function defined by Equation \eqref{eq:f}.  There exists a change of coordinates $z(\tilde{z})$ such that
$$ f_j(q,z(\tilde{z})) = \exp(g_j(\check{q}(q))) \cdot \sum_{l=1}^m \check{q}^{C_l-C_j}(q) \cdot \tilde{z}^{\bar{v}_l}. $$
\end{lemma}

\begin{proof}
Take the change of coordinates $\widetilde{\hat{z}_l} = \hat{z}_l \exp g_l(\check{q}(q))$ for $l=1,\ldots,n$.  Then $\widetilde{\hat{z}}^{v} = \hat{z}^v \exp \left( \sum_{i=1}^n a_i g_i(\check{q}(q)\right)$ for $v = \sum_{i=1}^n a_i v_i \in N$.  In particular this gives a change of coordinates for $z_i = \hat{z}^{e_i}$, $i=1,\ldots,n-1$.

Then for $l=1,\ldots,m$, 
$$\tilde{z}^{\bar{v}_l} = \widetilde{\hat{z}}^{v_l - v_j} = \hat{z}^{v_l-v_j} \exp \left( \sum_{i=1}^n (v_{l,i} - v_{j,i}) g_i(\check{q}(q))\right) = z^{\bar{v}_l} \exp \left( \sum_{i=1}^n (v_{l,i} - v_{j,i}) g_i(\check{q}(q))\right).  $$
Hence 
\begin{align*}
\exp(g_j(\check{q}(q))) \cdot \sum_{l=1}^m \check{q}^{C_l-C_j}(q) \cdot \tilde{z}^{\bar{v}_l} &= \exp(g_j(\check{q}(q))) \cdot \sum_{l=1}^m \check{q}^{C_l-C_j}(q) \cdot z^{\bar{v}_l} \exp \left( \sum_{i=1}^n (v_{l,i} - v_{j,i}) g_i(\check{q}(q)) \right) \\
&= \exp(g_j(\check{q}(q))) \cdot \sum_{l=1}^m q^{C_l-C_j}(q) \exp(g_l(\check{q}(q)) - g_j(\check{q}(q))) \cdot z^{\bar{v}_l} \\
&= \sum_{l=1}^m \left(q^{C_l-C_j}(q) \exp g_l(\check{q}(q))\right) z^{\bar{v}_l}\\
&= f_j(q,z)
\end{align*}
where the second equality follows from the definition of the mirror map
$$\check{q}^{C_l-C_j}(q) = q^{C_l-C_j} \exp\left(g_l - g_j - \sum_{i=1}^n (v_{l,i} - v_{j,i}) g_i(\check{q}(q))\right)$$
and the last equality follows from Theorem \ref{thm:open=g} that $1 + \delta_l(q) = \exp g_l(\check{q}(q))$.
\end{proof}

Now we classify the curve classes appeared in the hypergeometric functions $g_i$ (see Equation \eqref{eq:g}).
\begin{lemma} \label{lem:C}
For $C \in H_2(X)$, $C \cdot D_j < 0$ and $C \cdot D_l \geq 0$ for all $l \not= j$ if and only if
$$C = \sum_{\substack{i=1\\i\not=j}}^m a_i (C_i - C_j)$$
for some $a_i \geq 0$, $i \in \{1,\ldots,m\} - \{j\}$ satisfying $\sum_{i\not=j} a_i \bar{v}_i = 0$.  Here $\beta_i \in H_2(X,T)$ denotes the basic disc class corresponding to the toric divisor $D_i$ of $X$.
\end{lemma}

\begin{proof}
If $C \cdot D_j < 0$ and $C \cdot D_l \geq 0$ for all $l \not= j$, then $C = \sum_{i=1}^m a_i \beta_i$ for $a_i \geq 0$ for $i\not=j$ and $a_j < 0$.  Moreover since $C \in H_2(X)$, $\partial C = \sum_{i=1}^m a_i v_i = 0$.  Then $a_j v_j = -\sum_{i\not=j} a_i v_i$, and so $a_j = a_j \unu(v_j) = -\sum_{i\not=j} a_i \unu(v_i) = -\sum_{i\not=j} a_i$.  Hence $C = \sum_{i\not=j} a_i (\beta_i - \beta_j)$, and $\sum_{i\not=j} a_i (v_i - v_j) = 0$.  Now $C_l = \beta_l - \sum_{k=1}^n v_{l,k} \beta_k$ for any $l$.  Thus $\beta_i - \beta_j = C_i - C_j + \sum_{k=1}^n (v_{i,k} - v_{j,k}) \beta_k$ and
$$ C = \sum_{i\not=j} a_i (\beta_i - \beta_j) = \sum_{i\not=j}^m a_i (C_i - C_j) + \sum_{k=1}^n \left(\sum_{i\not=j} a_i (v_{i,k} - v_{j,k})\right) \beta_k = \sum_{i\not=j} a_i (C_i - C_j).$$

Conversely, if $C = \sum_{i\not=j}^m a_i (C_i - C_j) \in H_2(X)$ for some $a_i \geq 0$, $i \in \{1,\ldots,m\} - \{j\}$ satisfying $\sum_{i\not=j} a_i \bar{v}_i = 0$, then $C = \sum_{i\not=j} a_i (\beta_i - \beta_j)$.  Thus $C \cdot D_j < 0$ and $C \cdot D_l \geq 0$.
\end{proof}

Using the above two lemmas, we can then prove that $f_j$ satisfies the normalization condition of slab functions.
\begin{prop} \label{prop:key}
$\log f_j(q,z)$ has no term of the form $a\, q^{C}$ where $a \in \C$ and $C \in H_2(X)-\{0\}$.
\end{prop}

\begin{proof}
By Lemma \ref{lem:change}, $f_j(q,z(\tilde{z}))$ can be written as $\exp(g_j(\check{q}(q))) \cdot \sum_{l=1}^m \check{q}^{C_l-C_j}(q) \cdot z^{\bar{v}_l}$ up to a change of coordinates in $z$.  Since a change of coordinates in $z$ does not affect whether $\log f_j(q,z)$ has term of the form $a\, q^{C}$ or not, it suffices to prove that $\log$ of this expression has no term of the form $a\, q^{C}$.

We will prove that the part which has no $z$ dependence of 
$$ \log \left( 1 + \sum_{\substack{l=1\\l\not=j}}^m \check{q}^{C_l-C_j} \cdot z^{\bar{v}_l} \right) = \sum_{p>0} \frac{(-1)^{p-1}}{p} \left(\sum_{\substack{l=1\\l\not=j}}^m \check{q}^{C_l-C_j} \cdot z^{\bar{v}_l} \right)^p$$
is $-g_j(\check{q})$, and hence
$$ \log \left(\exp(g_j(\check{q})) \cdot \sum_{l=1}^m \check{q}^{C_l-C_j} \cdot z^{\bar{v}_l}\right) = g_j(\check{q}) + \log \left(\sum_{l=1}^m \check{q}^{C_l-C_j} \cdot z^{\bar{v}_l}\right)$$
has no term of the form $c \cdot \check{q}^C$.  Taking the inverse mirror map $\check{q}(q)$ (which does not affect whether the expression has terms with no $z$-dependence or not), we obtain the required statement.

A term with no $z$-dependence in $\left(\sum_{\substack{l=1\\l\not=j}}^m \check{q}^{C_l-C_j} \cdot z^{\bar{v}_l} \right)^p$ takes the form $c \cdot \check{q}^{\sum_{l\not=j} a_l (C_l-C_j)}$ where $\sum_{l\not=j} a_l \bar{v}_l = 0$.  By Lemma \ref{lem:C}, it corresponds to
$$C = \sum_{l\not=j} a_l (C_l-C_j)$$
with the property that $C \cdot D_j < 0$ and $a_i = C \cdot D_i \geq 0$ for all $i \not=j$.   Moreover $p = - D_j \cdot C = \sum_{l\not=j} a_l$.  

Here comes the key: $\check{q}^C$ appears $\frac{(-C \cdot D_j)!}{\prod_{l\not=j} (C \cdot D_l) !}$ times in the expansion of $\left(\sum_{\substack{l=1\\l\not=j}}^m \check{q}^{C_l-C_j} \cdot z^{\bar{v}_l} \right)^p$.  Hence the part with no $z$-dependence of
$$ \sum_{p>0} \frac{(-1)^{p-1}}{p} \left(\sum_{\substack{l=1\\l\not=j}}^m \check{q}^{C_l-C_j} \cdot z^{\bar{v}_l} \right)^p$$
is
$$ \sum_{\substack{C \cdot D_j < 0\\C \cdot D_l \geq 0}} \frac{(-1)^{(-C \cdot D_j)-1}}{(-C \cdot D_j)} \frac{(-C \cdot D_j)!}{\prod_{l\not=j} (C \cdot D_l)!} \check{q}^C = -g_j(\check{q}).$$
\end{proof}

\begin{remark}
Note that the part with no $z$-dependence of
$ \sum_{p>0} \frac{(-1)^{p-1}}{p} \left(\sum_{\substack{l=1\\l\not=j}}^m \check{q}^{C_l-C_j} \cdot z^{\bar{v}_l} \right)^p$ has a natural tropical meaning: it is a counting of tropical curves formed by unions of tropical discs corresponding to some vectors $\bar{v}_l$, and the requirement of no $z$-dependence is the balancing condition on $\bar{v}_l$'s.  This gives a tropical meaning of the hypergeometric functions $g_j$'s.  In order to make this intuitive interpretation rigorous, more foundational theory on counting tropical discs (and its relation with counting tropical curves) is needed.  Nevertheless our proof does not rely on this tropical interpretation.
\end{remark}

The remaining three conditions of being slab functions simply follow from the definition of $f_j$, and hence $f_j$ is a slab function.

\begin{prop}
If we set $f_{\cb,v_i} = f_i$ for all $i = 1,\ldots,m$ and $\cb \ni v_i$, then $f_{\cb,v_i}$ satisfies the defining properties of slab functions in Definition \ref{def:slab}.
\end{prop}

\begin{proof}
For $f_i(q,z) = \sum_{k=1}^m \left(\sum_{\alpha} q^\alpha n_{\beta_k+\alpha}\right) q^{C_k-C_i} z^{v_k-v_i}$, the constant term is obtained from the summand with $k=i$ and $\alpha=0$, which is $n_{\beta_k}=1$.  This proves Condition (1).

From the expression of $f_i$ we see that $ f_{i} = q^{C_j-C_i} z^{v_j-v_i} f_{j}.$  Thus Condition (2) follows.

Since we set $f_{\cb,v_i} = f_i$, $f_{\cb,v}$ does not depend on $\cb$ which is Condition (4).

Condition (3) is proved in Proposition \ref{prop:key}.
\end{proof}

This finishes the proof of Theorem \ref{thm:main}.  From Equation \eqref{eq:X^tg} and \eqref{eq:X^sg}, it follows immediately that the mirror constructed using tropical geometry by Gross-Siebert equals to the mirror constructed using symplectic geometry by Chan-Lau-Leung, which is Corollary \ref{cor:same}.

\section{Examples}

\begin{example}
$X = K_{\bP^1}$.  Let $\sigma$ be the closed interval $[-1,1] \subset \R$, and take the triangulation $\cP$ given by cutting the interval into two pieces $[-1,0]$ and $[0,1]$.  The fan is obtained by coning over the triangulation of $\sigma \times \{1\}$, and the generators of rays of the fan are $v_1 = (1,1),v_2=(0,1)$ and $v_3=(-1,1)$.  All effective curve classes are multiples of the exceptional curve class $C$, which corresponds to the K\"ahler parameter $q=q^C$ and mirror complex parameter $\check{q} = \check{q}^C$.  See Figure \ref{fig:KP1}.

\begin{figure}[htb!]
\begin{center}
\includegraphics{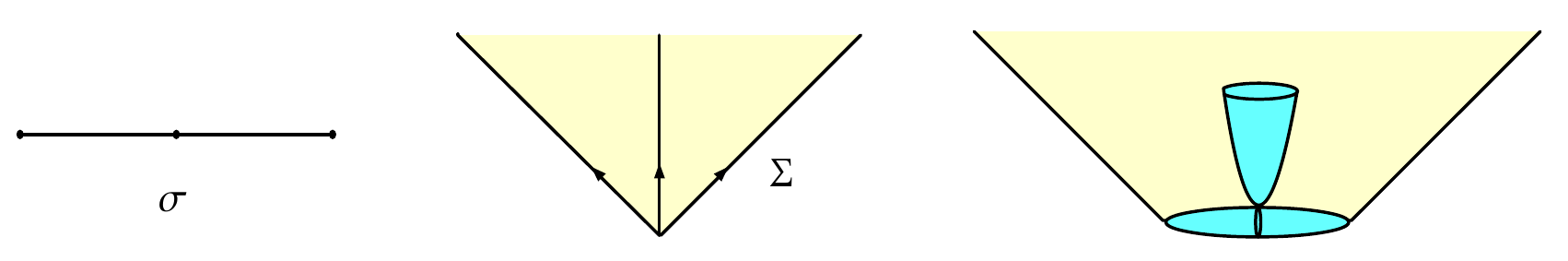}
\caption{The fan (in the middle) and moment-map polytope (on the right) of $K_{\bP^1}$.  Non-trivial open Gromov-Witten invariants come from disc classes of the form $\beta_2+\alpha$ shown on the right.}\label{fig:KP1}
\end{center}
\end{figure}

In the following we will consider the wall-crossing generating function $f_j$ for $j=2$, which is
$$ f_2 = \left(1+\sum_{\alpha \not= 0} n_{\beta_2+\alpha} q^{\alpha} \right) + x + q x^{-1}.$$
In this case it turns out that $\sum_{\alpha \not= 0} n_{\beta_2+\alpha} q^{\alpha} = q$ (which follows from computing the inverse mirror map).

By open mirror theorem and the mirror map $q = \check{q}(q) \exp 2 g(\check{q}(q))$, we have
$$f_2 = \exp g(\check{q}(q)) + x + \check{q}(q) \exp 2 g(\check{q}(q)) x^{-1}$$
where
$$ g(\check{q}) = \sum_{l>0} (-1)^{2l} \frac{(2l-1)!}{(l!)^2} \check{q}^l = \sum_{l>0} \frac{(2l-1)!}{(l!)^2} \check{q}^l$$
and $q = \check{q} \exp (2 g(\check{q}))$ is the mirror map.
We show that it satisfies the normalization condition (3).

\begin{align*}
\log f_2 = \log (\exp g(\check{q}) + x + \check{q} \exp 2 g(\check{q}) x^{-1}) &= g(\check{q}) + \log (1 + \tilde{x} + \check{q} \tilde{x}^{-1}) \\
&= g(\check{q}) + \sum_{k>0} \frac{(-1)^{k-1}}{k} (\tilde{x} + \check{q} \tilde{x}^{-1})^k
\end{align*}
where $\tilde{x} := x \exp (-g(\check{q}))$.  The summands of the second term have constant part (which has no dependence on $x$) only when $k=2l$ for some $l>0$.  Thus the constant part of the above expression is
$$ g(\check{q}) + \sum_{l>0} \frac{(-1)^{2l-1}}{2l} \frac{(2l)!}{(l!)^2} \check{q}^l = 0. $$
Hence the wall-crossing generating function for open Gromov-Witten invariants satisfies Gross-Siebert normalization condition.
\end{example}

\begin{example}
$X = K_{\bP^2}$.  Let $\sigma$ be the lattice polytope in $\R^2$ whose vertices are $(1,0),(0,1)$ and $(-1,-1)$, and take the triangulation $\cP$ given by dividing $\sigma$ into three pieces using the interior lattice point $(0,0)$.  The fan is obtained by coning over the triangulation of $\sigma \times \{1\}$, and the generators of rays of the fan are $v_1 = (1,0,1),v_2=(0,1,1)$, $v_3=(0,0,1)$.  All effective curve classes are multiples of the exceptional curve class $C$ (which is the line class in $\bP^2$), which corresponds to the K\"ahler parameter $q=q^C$ and mirror complex parameter $\check{q} = \check{q}^C$.  See Figure \ref{fig:KP2}.

\begin{figure}[htb!]
\begin{center}
\includegraphics[scale=0.9]{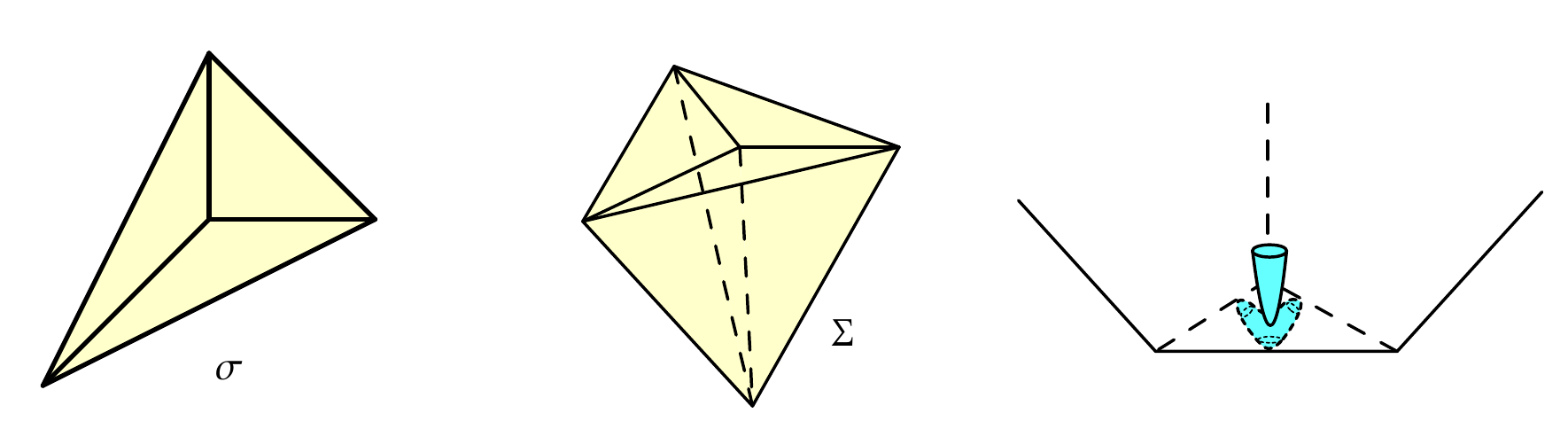}
\caption{The fan (in the middle) and moment-map polytope (on the right) of $K_{\bP^2}$.  Non-trivial open Gromov-Witten invariants come from disc classes of the form $\beta_3+\alpha$ shown on the right.}\label{fig:KP2}
\end{center}
\end{figure}

In the following we will consider the wall-crossing generating function $f_j$ for $j=3$, which is
$$ f_3 = \left(1+\sum_{\alpha \not= 0} n_{\beta_3+\alpha} q^{\alpha} \right) + x + y + q x^{-1}y^{-1}.$$

By open mirror theorem and the mirror map $q = \check{q}(q) \exp 3 g(\check{q}(q))$, we have
$$ f_3 = \exp g(\check{q}(q)) + x + y + \check{q}(q) \exp 3 g(\check{q}(q)) x^{-1} y^{-1}$$
where
$$ g(\check{q}) = \sum_{l>0} (-1)^{3l} \frac{(3l-1)!}{(l!)^3} \check{q}^l $$
and $q = \check{q} \exp (3 g(\check{q}))$ is the mirror map.  In this case $1+\sum_{\alpha \not= 0} n_{\beta_3+\alpha} q^{\alpha}=\exp g(\check{q}(q))$ is an infinite series
$$1 - 2q + 5 q^2 - 32 q^3 + 286 q^4 - 3038 q^5 + 35870 q^6 - \ldots.$$
We show that $f_j$ satisfies the normalization condition (3) in Definition \ref{def:slab}.

\begin{align*}
\log f_3 = \log (\exp g(\check{q}) + x + y + \check{q} \exp 3 g(\check{q}) x^{-1}y^{-1}) &= g(\check{q}) + \log (1 + \tilde{x} + \tilde{y} + \check{q} \tilde{x}^{-1}\tilde{y}^{-1}) \\
&= g(\check{q}) + \sum_{k>0} \frac{(-1)^{k-1}}{k} (\tilde{x} + \tilde{y} + \check{q} \tilde{x}^{-1}\tilde{y}^{-1})^k
\end{align*}
where $\tilde{x} := x \exp (-g(\check{q}))$ and $\tilde{y} := y \exp (-g(\check{q}))$.  The summands of the second term have constant part (which has no dependence on $(x,y)$) only when $k=3l$ for some $l>0$.  Thus the constant part of the above expression is
$$ g(\check{q}) + \sum_{l>0} \frac{(-1)^{3l-1}}{3l} \frac{(3l)!}{(l!)^3} \check{q}^l = 0. $$
Hence the wall-crossing generating function for open Gromov-Witten invariants satisfies Gross-Siebert normalization condition.
\end{example}

\bibliographystyle{amsalpha}
\bibliography{geometry}

\end{document}